\documentclass[english]{article}

\title{Hilbert's Tenth Problem for rational function fields over p-adic fields}
\author{
Claudia Degroote\thanks{Ph.~D.~fellow of the Research Foundation --- Flanders (FWO).
\textbf{Address:} Ghent University, Department of Mathematics, Kr{\ij}gs\-laan 281, 9000 Gent, Belgium.
\textbf{E-mail:} cdegroote@cage.ugent.be.}
\and
Jeroen Demeyer\thanks{Postdoctoral Fellow of the Research Foundation --- Flanders (FWO).
\textbf{Address:} Ghent University, Department of Mathematics, Kr{\ij}gs\-laan 281, 9000 Gent, Belgium.
\textbf{E-mail:} jdemeyer@cage.ugent.be.}
}

\usepackage{babel}
\usepackage{latexsym,amssymb,amsmath,amsthm}
\usepackage{paralist}
\usepackage{txfonts}

\newcommand{\N}{\mathbb{N}}
\newcommand{\R}{\mathbb{R}}
\newcommand{\Q}{\mathbb{Q}}
\newcommand{\Z}{\mathbb{Z}}
\newcommand{\Ok}{\mathcal{O}}
\newcommand{\I}{\mathcal{I}}
\newcommand{\Leg}[2]{\left(\frac{#1}{#2}\right)}
\newcommand{\al}{\alpha}
\newcommand{\eps}{\varepsilon}
\newcommand{\pr}{\mathfrak{p}}
\newcommand{\pf}[2]{\langle #1,#2 \rangle}
\newcommand{\form}[1]{\langle #1 \rangle}
\newcommand{\squares}[1]{#1^{\ast2}}
\newcommand{\sqc}[1]{#1^{\ast}/\squares{#1}}
\newcommand{\inv}{^{-1}}
\newcommand{\orthosum}{\mathbin{\bot}}
\DeclareMathOperator{\degt}{\deg^{\!\dagger}}

\newtheorem{st}{Theorem}[section]
\newtheorem{prop}[st]{Proposition}
\newtheorem{lem}[st]{Lemma}
\newtheorem{gev}[st]{Corollary}

\theoremstyle{definition}
\newtheorem{defi}[st]{Definition}

\theoremstyle{remark}
\newtheorem*{bew}{\normalfont\bfseries Proof}

\theoremstyle{remark}
\newtheorem*{acknowledgements}{\normalfont\itshape Acknowledgements}

\begin{document}

\maketitle

\begin{abstract}
Let $K$ be a p-adic field (a finite extension of some $\Q_p$)
and let $K(t)$ be the field of rational functions over $K$.
We define a kind of quadratic reciprocity symbol for polynomials over $K$ and
apply it to prove isotropy for a certain class of quadratic forms over $K(t)$.
Using this result, we give an existential definition for the predicate ``$\varv_t(x) \geq 0$'' in $K(t)$.
This implies undecidability of diophantine equations over $K(t)$.
\end{abstract}

\section{Introduction}

In \cite{kim-roush-padic}, Kim and Roush proved undecidability for diophantine equations
for rational function fields over a subfield of a p-adic field of odd residue characteristic.
Our interest went out to improving their methods so they would also work in the case of residue characteristic 2.
While the present paper follows some of the structure of Kim and Roush's proof,
our proof is a general one that handles both $p$ odd and $p=2$.
We also simplify many of the methods of Kim and Roush
by working more in the context of the theory of quadratic forms. 
However, we only deal with p-adic fields (as opposed to subfields of p-adic fields).

This result fits in with other results concerning Hilbert's Tenth Problem.
In his famous list of 23 problems, Hilbert asked for an algorithm that solves the following question:
given a polynomial with integer coefficients in any number of unknowns, does this polynomial have an integer zero?
In 1970 Matiyasevich proved, building on earlier work of Davis, Putnam and Robinson,
that recursively enumerable sets are Diophantine (called the DPRM theorem).
From this it followed that there is no algorithm to decide whether a polynomial 
over the integers has integer zeros.
The undecidability of diophantine equations has been shown for many other rings and fields,
\cite{pheidas-zahidi} gives an overview of what is known.

For rational function fields $K(t)$, to prove diophantine undecidability
(i.e.~a negative answer to Hilbert's Tenth Problem)
it suffices to give an existential definition of the valuation ring at $t$.
This method was first used by Denef (\cite{denef-real}) in characteristic zero,
in general the result can be stated as follows:

\begin{st}(\cite[Theorem 2.3]{pheidas-zahidi})\label{denef}
Let $K$ be a field and let $K_0$ denote the prime subfield of $K$.
Let $t$ be a transcendental element over $K$.
Suppose that there exists an existential formula $\psi(x)$ such that the following hold:
\begin{enumerate}
\item for every $x \in K_0(t)$ such that $\varv_t(x) \geq 0$, $\psi(x)$ holds;
\item for every $x \in K(t)$ such that $\psi(x)$ holds, we have $\varv_t (x) \geq 0$.
\end{enumerate}
Then the existential theory of $K(t)$ is undecidable. 
\end{st}

Therefore, the aim of this paper is really to give an existential definition
of the predicate ``$\varv_t(x) \geq 0$''
(we do not need to restrict to the prime field as in Theorem~\ref{denef}).
This can easily be reduced to giving an existential definition of ``$\varv_t(x)$ is even'',
this reduction is done implicitly in Theorem~\ref{diophdef}.
We will use quadratic forms to define ``$\varv_t(x)$ is even''.

In section \ref{sec-definitions} of this paper,
we start with some basic definitions and theorems about quadratic forms,
we give a refined Dirichlet density theorem for global fields
and we introduce Newton polygons which immediately give
the valuations of the zeros of a polynomial.

At various places in their proof, Kim and Roush
take an element of $K(t)$ (where $K$ is a p-adic field), apply a variable transformation
which might live over a finite extension of $K$ and then go the residue field
(this is the rational function field over a finite field).
The resulting functions are called ``edge functions''.
They then reason with these functions and lift back to $K(t)$.
Since our proof also needs to work for even residue characteristic
and we want to work with quadratic forms, we cannot do this anymore.
Instead, in section \ref{quadr-recipr},
we define a kind of quadratic reciprocity symbol for polynomials over $K$.
In the case of odd residue charactertic, this symbol can be defined
completely in terms of the residue field.
However, it is actually a lot more natural not to look at the residue field at all.
We also prove a quadratic reciprocity law for this symbol.
Section \ref{isotr-qf} contains our main result regarding quadratic forms:
if $f \in K[t]$ has a particular Newton polygon,
then certain quadratic forms over $K(t)$ involving $f$ and an unknown function $s \in K(t)$
can be made isotropic.
This is then used in section \ref{dioph-undec} to give an existential definition
of ``$\varv_t(x) \geq 0$''.
Also in this last section, we use an elliptic curve to
give an existential definition of the constants $K$ in $K(t)$.

\section{Definitions}\label{sec-definitions}

We start with some definitions and properties of quadratic forms.
We state Milnor's exact sequence, giving a local-global principle
for the Witt ring $W(K(t))$ of a rational function field over a general base field $K$,
and the well known fact that a Pfister form is isotropic if and only if it is hyperbolic.
We assume that the reader is familiar with the basic theory of quadratic forms,
in particular the Witt ring.
We refer to \cite{lam-qf} or \cite{scharlau-qhf}.

\begin{defi}
Let $K$ be a field of characteristic $\neq 2$. For $\al_1, \ldots, \al_n \in K^*$, the quadratic form
\[
	\bigotimes_{i=1}^n \form{1,\al_i} = \form{1,\al_1,\al_2,\ldots,\al_n,\al_1 \al_2,\ldots,\al_1 \al_2 \cdots \al_n}
\]
is called an \emph{$n$-fold Pfister form}.
\end{defi}

\begin{st}(\cite[Ch.~X, Theorem 1.7]{lam-qf})\label{pfister-hyperbolic}
Let $K$ be a field of characteristic $\neq 2$, $\varphi$ a Pfister form over $K$.
If $\varphi$ is isotropic, then $\varphi=0$ in the Witt ring $W(K)$.
\end{st}

Let $p(t)$ be a monic, irreducible polynomial over $K$.
An arbitrary quadratic form $\varphi$ of dimension $n$ over $K(t)$ can be written as $\varphi_1 \orthosum \form{p(t)} \varphi_2$ with $\varphi_1 = \form{u_1,\ldots,u_r}$ and $\varphi_2 = \form{u_{r+1},\ldots,u_n}$ where the $u_i(t)$ are polynomials coprime with $p(t)$.
Denote the reduction of a polynomial $u(t)$ modulo $p(t)$ with $\overline{u(t)}$.
Then $\overline{\varphi_1} = \form{\overline{u_1},\ldots,\overline{u_r}}$ and $\overline{\varphi_2} = \form{\overline{u_{r+1}},\ldots,\overline{u_n}}$ are called the \emph{first} and \emph{second residue forms of $\varphi$}.
We denote the \emph{second residue class map}
\[W(K(t)) \to W(K[t]/(p)): \varphi \mapsto \overline{\varphi_2}\]
with $\delta_p$.

\begin{st}[Milnor exact sequence]\label{milnorexactsequence}
Let $K$ be a field of characteristic $\neq 2$. 
Let $i$ be the functorial map $W(K)\to W(K(t))$.
Let $\delta = \bigoplus \delta_p$ where the direct sum extends over all monic irreducible polynomials $p(t) \in K[t]$.
Then the following sequence of abelian groups is split exact:
\[0 \to W(K) \stackrel{i}{\to} W(K(t)) \stackrel{\delta}{\to} \bigoplus_p W(K[t]/(p)) \to 0. \]
\end{st}

\begin{bew}
See \cite[Ch.~IX, Theorem 3.1]{lam-qf}.
\end{bew}

At some point, we will construct a polynomial over a p-adic field $K$ with certain properties. 
To do this, we will start from a polynomial in the reduction $k[t]$, 
where $k$ is the (finite) residue field of $K$, 
which we will find using the following generalization of Dirichlet's density theorem:

\begin{st}\cite[Theorem A.10 with $S_0 = \emptyset$]{bass-milnor-serre}\label{dichtheid}
Let $F$ be a global field, $S_\infty$ a finite non-empty set of primes of $F$,
containing all archimedean primes when $F$ is a number field.
Let
\[A=\{x\in F : \varv_\pr(x) \geq 0 \text{ for all } \pr\notin S_\infty\}.\]
Suppose we are given $a,b\in A$ such that $a A + b A = A$
and for each $\pr\in S_\infty$ an open subgroup $V_\pr\subset F^\ast_\pr$
and an $x_\pr\in F^\ast_\pr$.
Suppose also that
$V_\pr$ has finite index in $F^\ast_\pr$ for at least one $\pr \in S_\infty$.

Then there exist infinitely many primes $\pr_0\notin S_\infty$
such that there is a $c\in A$ satisfying \[c\equiv a\mod b\]
\[c\in x_\pr V_\pr \quad\mbox{for all}\quad \pr\in  S_\infty\]
\[c A=\pr_0.\]
\end{st}

Now we look at valued fields and we define the Newton polygon
of a polynomial over a valued field.
For the theory of valuations, we refer to \cite{engler-prestel}.
Let $K$ be a field with a discrete henselian valuation $\varv$.
Let $\Ok$ denote the valuation ring.
The fact that $K$ is ``henselian'' means that the following holds:

\begin{st}[Hensel's Lemma]\cite[Theorem~4.1.3]{engler-prestel}\label{hensel}
For all $f\in\Ok[t]$ and $a\in\Ok$ such that $\varv(f(a)) > 2\varv(f'(a))$,
there exists a $b\in\Ok$ such that $f(b)=0$ and $\varv(b-a)>\varv(f'(a))$.
\end{st}

\begin{defi}
Let $f(t) = a_0 + a_1 t + \ldots + a_d t^d$ be a polynomial over $K$ with $a_0 a_d \neq 0$.
The \emph{Newton polygon} of $f$ is the lower convex hull of the points $(i,\varv(a_i))$ in $\R^2$.
A \emph{vertex} of the Newton polygon is a point $(i,\varv(a_i))$
where two edges of a different slope meet.
We call $i$ the \emph{degree} of the vertex $(i,\varv(a_i))$.
\end{defi}

The Newton polygon of a polynomial consists of a sequence of edges
with strictly increasing slopes, for which the following holds:

\begin{st}(\cite[II~(6.3), II~(6.4)]{neukirch-az})
Let $K$ be a field with a discrete henselian valuation $\varv$.
Let $f(t) = a_0 + a_1 t + \ldots + a_d t^d$ be a polynomial over $K$ with $a_0 a_d \neq 0$.
Denote the unique extension of $\varv$ on the splitting field of $f$ also by $\varv$.
If $(r,\varv(a_r))$~-- $(s,\varv(a_s))$
is an edge of the Newton polygon of $f$ with slope $m$, then $f$ has exactly $s-r$ roots
$\alpha_1, \ldots, \alpha_{s-r}$ with valuation $\varv(\alpha_1) = \dots = \varv(\alpha_{s-r}) = -m$.
If the slopes of the Newton polygon of $f$ are $m_1 < \ldots < m_k$ then
\begin{equation}\label{fact-sl}
	f(t) = a_d \prod_{j=1}^k f_j(t)
,\end{equation}
with $f_j(t) = \prod_{\varv(\alpha_i)=-m_j}(t-\alpha_i)\in K[t]$.
\end{st}

\begin{defi}
We define \emph{the factorization of $f$ according to the slopes}
to be the expression $a_d \prod_{j=1}^k f_j(t)$ in \eqref{fact-sl}.
Note that the polynomials $f_j(t)$ are not necessarily irreducible over $K$.
The Newton polygon of each $f_j(t)$ has exactly one edge of slope $m_j$.
\end{defi}

\begin{lem}\label{same-sq-cl}
Let $a, b \in K^*$. 
If $\varv(b) > \varv(a) + \varv(4)$, then $a + b = x^2 a$ for some $x \in K^*$.
\end{lem}

\begin{bew}
By assumption, $\varv(a\inv b) > 0$, so $1 + a\inv b \in \Ok$.
Let $f(t) = t^2 - (1 + a\inv b)$.
Then $\varv(f(1)) = \varv(b) - \varv(a) > \varv(4) = 2 \varv(f'(1))$.
From Hensel's Lemma follows that $1 + a\inv b$ is a square.
\end{bew}

\begin{lem}\label{termvaluatie}
Let $f(t) = \sum_{k=0}^d b_k t^k$ be a monic polynomial (i.e.~$b_d = 1$)
whose Newton polygon has only one edge.
Let $m := -\varv(b_0)/d$ be the slope.
Let $\alpha \in K$.
Then
\[
	\varv\big(b_k \alpha^k\big)
		\geq k\big(m + \varv(\alpha)\big) - d m
		= (k - d)\big(m + \varv(\alpha)\big) + d \varv(\al)
.\]
\end{lem}

\begin{bew}
This follows immediately from the fact that the Newton polygon
of $f$ has only one edge.
\end{bew}

\begin{prop}\label{wortelinvullen}
Let $K$ be a field with a discrete henselian valuation $\varv$ and let $\al \in K$.
Let $f$ be a polynomial over $K$ of even degree with $f(0) \neq 0$.
Assume that the Newton polygon of $f$ has only one edge, let $m$ be the slope.
Assume that $m \neq -\varv(\alpha)$ and let $N \in \N$ be such that
\begin{equation}\label{ongelijkheidN}
	N > \frac{\varv(4)}{|m + \varv(\alpha)|}
.\end{equation}
Assume that $f$ is of the form
\[f = a(t) + g(t) t^N + z(t) t^{2 N + \deg g - \deg z},\]
where $a$, $g$ and $z$ are polynomials over $K$.
Assume that $\deg(g)$ and $\deg(z)$ are even and that $\deg(a) < N$ and $\deg(z) < N$.

If $m < -\varv(\al)$, then $f(\al) = z(\al)$ in $\sqc{K}$.
If $m > -\varv(\al)$, then $f(\al) = a(\al)$ in $\sqc{K}$.
\end{prop}

Remark that we can always take $N = 1$ in \eqref{ongelijkheidN}
if the residue characteristic is different from $2$.

\begin{bew}
Let $d = 2N + \deg(g)$ be the degree of $f$.
Suppose first that the slope $m$ of $f$ is strictly smaller than $-\varv(\al)$.
Let $c := z(\al) \al^{2 N + \deg g - \deg z}$,
whose leading term is the monomial of strictly lowest valuation in $f(\al)$.
We then have that
\[
	c\inv f(\al) = c\inv a(\al) + c\inv g(\al) \al^N + 1
.\]
Using Lemma~\ref{termvaluatie}, the inequality \eqref{ongelijkheidN} implies
\[
	\varv(a(\al) + g(\al) \al^N)
		\geq (-N)(m + \varv(\al)) + d \varv(\al)
		= N |m + \varv(\al)| + \varv(c)
		> \varv(4) + \varv(c)
.\]
Therefore,
\[
	\varv(c\inv f(\al) - 1) > \varv(4)
.\]
By Hensel's Lemma applied to the polynomial $x^2 - (c\inv f(\al))$,
we find that $c\inv f(\al)$ is a square.
Since $\al^{2 N + \deg g - \deg z}$ is a square,
it follows that $z(\al)$ is in the same square class as $f(\al)$.

If the slope $m$ of $f$ is strictly bigger than $-\varv(\al)$, we let $c := a(\al)$.
The constant term of $c$ is the monomial of strictly lowest valuation in $f(\al)$.
Note that $\varv(c) = -m d$.
We then have that
\[
	c\inv f(\al) = 1 + c\inv g(\al) \al^N + c\inv z(\al) \al^{d - \deg z}
.\]
Using Lemma~\ref{termvaluatie}, the inequality \eqref{ongelijkheidN} implies
\[
	\varv(g(\al) \al^N + z(\al) \al^{d - \deg z})
		\geq N (m + \varv(\al)) - d m
		= N |m + \varv(\al)| + \varv(c)
		> \varv(4) + \varv(c)
.\]
Therefore,
\[
	\varv(c\inv f(\al) - 1) > \varv(4)
.\]
As before, we find that $c\inv f(\al)$ is a square;
hence $f(\al)$ is in the same square class as $c = a(\al)$.
\end{bew}

\section{A quadratic reciprocity symbol for polynomials}\label{quadr-recipr}

From now on, let $K$ be a p-adic field, that is a finite extension of some $\Q_p$.
Fix a uniformizer $\pi$ of $K$, i.e.~a generator of the maximal ideal in the valuation ring $\Ok$.
We normalize the valuation on $K$ and all its finite extensions
such that $\varv(\pi) = 1$.
This means for example that $\varv(\sqrt\pi) = 1/2$ in $K(\sqrt\pi)$.
Remark that $\varv$ extends uniquely to the algebraic closure of $K$.

The structure of the Witt ring of p-adic fields is well-known,
in particular we know that $I^2(K) \cong \Z/2\Z$ and that $I^3(K) = 0$ (see \cite[Ch.~VI, Corollary 2.15]{lam-qf}).
In this section, we will define a kind of Legendre symbol for the function field $K(t)$.
This symbol can be seen as the second residue map of a certain $3$-fold Pfister form over $K(t)$,
and takes two values according to the isotropy of this quadratic form.
We will prove multiplicativity and a quadratic reciprocity law for this symbol.

\begin{defi}
Let $q(t)$ be a monic irreducible polynomial over $K$ and
let $\alpha$ be a root of $q$ in the algebraic closure.
For $p(t)\in K[t]$, coprime to $q(t)$, we define the Legendre symbol
\begin{align*}
	\Leg{p}{q} &= \delta_q(\form{1,\pi}\form{1, -p(t)}\form{1, -q(t)}) \\
		&= \form{1,\pi}\form{1, -p(\alpha)} \in I^2(K(\alpha)) \cong \Z/2\Z
.\end{align*}
The second equality is justified because 
$\form{-1} \varphi = \varphi$ for $\varphi \in I^2(L)$.
We denote this symbol multiplicatively with values in $\{-1,1\}$,
despite the fact that the operation corresponds to addition of quadratic forms in the Witt ring.
This symbol is well defined for $p(t) \in (K[t]/q(t))^*$.
\end{defi}

% This symbol depends on the choice of uniformizer $\pi$,
% but this dependency on $\pi$ is different for residue characteristic odd and even.
% For $p \neq 2$, the quadratic form $\form{1,\pi} \form{1,-u}$ is anisotropic for any unit $u$ that is not a square in $K^*$. 
% So if $p(\al)$ is a unit, then $\form{1,\pi} \form{1, -p(\al)}$ is independent of the uniformizer $\pi$,
% but if $p(\al)$ has odd valuation, then for two different uniformizing elements $\pi$ and $\pi'$,
% $\form{1,\pi} \form{1, -p(\al)}$ and $\form{1,\pi'} \form{1, -p(\al)}$ can be different elements of $I^2(K(\alpha))$.
% For $p = 2$, it is no longer true that $\form{1,\pi} \form{1,-u}$ is anisotropic
% for any unit $u$ that is not a square, so regardless what the valuation of $p(\al)$ is,
% $\form{1,\pi} \form{1, -p(\al)}$ and $\form{1,\pi'} \form{1, -p(\al)}$ can be different.
% To be consistent, we will work all the time with one fixed uniformizer.

For every $n \in \Z$, we have $\Leg{p}{q} = \Leg{\pi^n p}{q}$ because of the factor $\form{1,\pi}$
in the definition of the Legendre symbol.
For odd residue characteristic,
this implies the following equivalent definition of the symbol:
\[
	\Leg{p}{q} = 1
	\iff
	\pi^{-\varv(p(\al))} p(\al) \text{ is a square in } K(\al)^*
,\]
where $\al$ is a root of $q(t)$.
For even residue characteristic, we cannot give such an easy equivalence.
Note that the symbol clearly depends on the choice of uniformizer $\pi$.
To be consistent, we will work all the time with one fixed uniformizer.

Next, we prove multiplicativity of the symbol.
\begin{prop}
Let $q(t)$ be a monic irreducible polynomial over $K$.
Let $p(t)$ and $r(t)$ be polynomials over $K$, coprime to $q(t)$.
Then
\begin{equation}\label{multiplicativity}
	\Leg{p r}{q} = \Leg{p}{q} \Leg{r}{q}
.\end{equation}
\end{prop}

\begin{bew}
Let $\al$ be a root of $q$.
The statement \eqref{multiplicativity} is equivalent to
\[
	\form{1,\pi} \form{1,-p(\al)} \orthosum \form{1,\pi} \form{1,-r(\al)} \orthosum -\form{1,\pi} \form{1,-p(\al) r(\al)} = 0
.\]
We can simplify this to
\begin{align*}
	\form{1,\pi} \form{1,1,-1,-p(\al),-r(\al),p(\al)r(\al)} &= 0 \\
	\form{1,\pi} \form{1,-1} \orthosum \form{1,\pi} \form{1,-p(\al),-r(\al),p(\al)r(\al)} &= 0 \\
	\form{1,\pi} \form{1,-p(\al)} \form{1,-r(\al)} &= 0
.\end{align*}
Since 3-fold Pfister forms are hyperbolic over $K$,
the last equality is always true.
\end{bew}

In order to further study this quadratic reciprocity symbol, we need to use transfers.
See \cite[Ch.~2, \S\,5]{scharlau-qhf} for the definition and properties of the transfer map.
\begin{prop}\label{transfer}
Let $K$ be a p-adic field and $L$ a finite extension of $K$. 
Let $s: L \to K$ be a non-zero $K$-linear map,
$s_*: W(L) \to W(K)$ the corresponding transfer map.
Then $s_*$ induces an isomorphism $I^2(L) ~\tilde{\to}~ I^2(K)$.
\end{prop}

\begin{bew}
Recall that $I^2(L)$ and $I^2(K)$ have $2$ elements,
so it suffices to prove that the non-zero element of $I^2(L)$
maps to the non-zero element of $I^2(K)$.
Let $\varphi$ be a 4-dimensional anisotropic Pfister form over $L$,
this means that $\varphi \neq 0$ in $I^2(L)$.
From \cite[Ch.~6, Theorem 4.4]{scharlau-qhf},
it follows that $s_*(\varphi)$ is Witt-equivalent to 
the unique anisotropic 4-dimensional quadratic form over $K$,
so $s_*(\varphi) \neq 0$ in $W(K)$. 
\end{bew}

\begin{prop}\label{overK}
Let $c \in K^*$ and $q \in K[t]$ be a monic irreducible polynomial.
Then
\[
	\Leg{c}{q} = \Leg{c}{t}^{\deg q}
.\]
\end{prop}

\begin{bew}
Let $\al$ be a root of $q$.
Let $\varphi = \form{1,\pi} \form{1,-c}$ considered in $W(K(\alpha))$.
Let $s: K(\al) \twoheadrightarrow K$ be a $K$-linear map.
Proposition~\ref{transfer} says that $\Leg{c}{q} = 1$
if and only if $s_*(\varphi) = 0$.
By \cite[Ch.~2, Theorem 5.6 and Lemma 5.8]{scharlau-qhf}, we have
\begin{align*}
	s_*(\varphi) &= \form{1,\pi} \form{1,-c} s_*(\form{1}_L) \quad \text{(in $W(K)$)} \\
		&=
		\left\{\begin{aligned}
			\form{1,\pi} \form{1,-c} & \quad \text{($\deg q$ odd)} \\
			0 & \quad \text{($\deg q$ even)}
		\end{aligned}\right.
.\end{align*}
This proves the proposition.
\end{bew}

Next, we want to find a quadratic reciprocity law.
For $p$ a monic irreducible polynomial of degree $n$ we define the $K$-linear map
\[s_p:K[t]/(p)\to K\]
by $s_p(1)=s_p(t)=\ldots=s_p(t^{n-2})=0$, $s_p(t^{n-1})=1$. This map gives us the transfer homo\-morphism
\[(s_p)_\ast:W(K[t]/(p))\to W(K).\]
For the prime at infinity, let $(s_\infty)_\ast=-id$. Then

\begin{st}\cite[Ch.~6, Theorem 3.5]{scharlau-qhf}\label{scharlau-exact-seq}
Let $K$ be a field of characteristic $\neq 2$.
Let $\delta_p:W(K(t))\to W(K[t]/(p))$ be the second residue class map
with respect to the irreducible polynomial $p$ or $t^{-1}$ in case $p=\infty$
and let $(s_p)_\ast$ be the transfer homomorphism given above.
For $\delta = \bigoplus \delta_p$ and $s_* = \sum(s_p)_*$ the following sequence is exact:
\[
	W(K(t)) \stackrel{\delta}{\to}
	\bigoplus_{p,\infty} W(K[t]/(p)) \stackrel{s_\ast}{\to}
	W(K) \to 0
.\] 
In particular,
\[\sum (s_p)_\ast\delta_p(\varphi)=0 \quad \text{(in $W(K)$)}\]
for every form $\varphi\in K(t)$.
\end{st}

Using this theorem, we can prove a reciprocity law for our symbol.

\begin{st}
Let $p(t)$ and $q(t)$ be monic irreducible polynomials over $K$.
Then
\begin{equation}\label{reciproc-law}
	\Leg{p}{q} = \Leg{-1}{t}^{\deg p \deg q} \Leg{q}{p}
.\end{equation}
\end{st}

\begin{bew}
Consider the quadratic form $\pf{1}{\pi}\pf{1}{-p(t)}\pf{1}{-q(t)}$.
The second residue class map applied to this form is trivial, except possibly at $p$, $q$ and $\infty$.
In those cases we have that
\[\delta_p (\pf{1}{\pi}\pf{1}{-p(t)}\pf{1}{-q(t)})=-\pf{1}{\pi}\pf{1}{-q(\beta)},\]
where $\beta$ is a root of $p$, and
\[\delta_q (\pf{1}{\pi}\pf{1}{-p(t)}\pf{1}{-q(t)})=-\pf{1}{\pi}\pf{1}{-p(\alpha)},\]
where $\alpha$ is a root of $q$.
We claim that
\[
	\delta_\infty\big(\form{1,\pi}\form{1,-p(t)}\form{1,-q(t)}\big)
		= \form{1,\pi} \form{-1, (-1)^{\deg p \deg q}}
.\]
Indeed, we find
\begin{align*}
	\delta_\infty\big(\form{1,\pi}\form{1,-p(t)}\form{1,-q(t)}\big)
		&= 0 = \form{1,\pi} \form{-1, 1}
		&& \text{($\deg p$ even, $\deg q$ even)},\\
	\delta_\infty\big(\form{1,\pi}\form{1,-p(t)}\form{1,-q(t)}\big)
		&= \form{1,\pi} \form{-1, 1}
		&& \text{($\deg p$ odd, $\deg q$ even)},\\
	\delta_\infty\big(\form{1,\pi}\form{1,-p(t)}\form{1,-q(t)}\big)
		&= \form{1,\pi} \form{-1, -1}
		&& \text{($\deg p$ odd, $\deg q$ odd)}.
\end{align*}
From Theorem \ref{scharlau-exact-seq}, it follows that 
\[(s_p)_\ast (\pf{1}{\pi}\pf{1}{-q(\beta)}) +(s_q)_\ast (\pf{1}{\pi}\pf{1}{-p(\alpha)})-\pf{1}{\pi}\pf{-1}{(-1)^{\deg p \deg q}}= 0.\]
By definition and multiplicativity of the symbol, we have that 
\[\pf{1}{\pi}\pf{1}{-(-1)^{\deg p \deg q}} = \Leg{-1}{t}^{\deg p \deg q}.\]
From Proposition~\ref{transfer} now follows
\[\Leg{p}{q} = \Leg{-1}{t}^{\deg p \deg q} \Leg{q}{p}.\]
\end{bew}

\section{Isotropy of a certain class of quadratic forms}\label{isotr-qf}

In this section, $K$ denotes a p-adic field with fixed uniformizer $\pi$.
Let $\Ok$ denote the valuation ring of $K$.

%\begin{lem}\label{isotroop}
%Let $K$ be a field of characteristic different from $2$ with a henselian valuation $\varv$.
%Let $q = \form{q_1, \dots, q_n}$ and $q' = \form{q_1+m_1, \dots, q_n+m_n}$ be quadratic forms over $K$ with $\varv(m_i)>\varv(q_i)+\varv(4)$.
%If $q$ is isotropic, then $q'$ is isotropic. 
%\end{lem}

%\begin{proof}
%Let $x_i \in K$ be such that $\sum_{i=1}^n q_ix_i^2 = 0$ and such that not all $x_i$ are zero. Without loss of generality we may assume that $\varv(q_1x_1^2)\leq\varv(q_ix_i^2)$ for all $i=2,\ldots,n$. Let
%\begin{equation}\label{isotroopy}
%	f(y)=y^2+\sum_{i\geq2} \frac{(q_i+m_i)x_i^2}{(q_1+m_1)x_1^2}
%.\end{equation}
%We will use Hensel's Lemma to find a zero of $f$ in $K$ near $y = 1$. Since $\varv(q_i+m_i)=\varv(q_i)$, we have $f(y) \in \Ok[y]$. Furthermore, $f(1) = (\sum m_i x_i^2)/((q_1+m_1)x_1^2)$, so
%\begin{align*}
%	\varv(f(1))
%		&\geq \min_{i} \left( \varv(m_i x_i^2) \right) - \varv(q_1 x_1^2)\\
%		&= \min_{i} \left( \varv(m_i) - \varv(q_i) + \varv(q_i x_i^2) - \varv(q_1 x_1^2) \right) \\
		%&\geq \min_{i} \left( \varv(m) - \varv(q_i) \right) \\
%		&> \varv(4)=2\varv(f'(1))
%.\end{align*}
%By Hensel's Lemma (Theorem~\ref{hensel}), $f$ has a zero $y \in \Ok$ satisfying $\varv(y-1) > \varv(2)$, so in particular $y \neq 0$. Now \eqref{isotroopy} implies
%\[(q_1+m_1)(x_1 y)^2 + \sum_{i\geq2}(q_i+m_i)x_i^2=0.\]
%Since $x_1 y \neq 0$, this means that $q'$ is isotropic.
%\end{proof}

\begin{lem}\label{ideaal}
Let $R$ be a ring and $x, y \in R$ such that $(x,y) = (1)$.
Let $\rho \in R^{*}$ and $N \in \N$.
Then $(x + \rho y^N, x y^N) = (1)$.
\end{lem}

\begin{proof}
Cubing the relation $(x,y) = (1)$, we get $(x^3, x^2 y, x y^2, y^3) = (1)$.
Clearly, $(x^3, x^2 y, x y^2, y^3) \subseteq (x^2, y^2)$; therefore, $(x^2, y^2) = (1)$.
We can continue this process by induction to get $(x^{2^n}, y^{2^n}) = (1)$ for all $n$,
so also $(x^2, y^{2 N}) = (1)$.
Let $\I := (x + \rho y^N, x y^N)$.
One can easily check that $x^2 = x (x + \rho y^N) - \rho x y^N \in \I$
and $y^{2 N} = \rho\inv y^N(x + \rho y^N) - \rho\inv x y^N \in \I$.
It follows that $(1) = (x^2, y^{2 N}) \subseteq \I$, hence $\I = (1)$.
\end{proof}

The following is the main theorem regarding isotropy of quadratic forms.

\begin{st}\label{maaks}
Let $\gamma \in K^*$.
Let $g \in K[t]$ with $g(0) \neq 0$.
If all the vertices of the Newton polygon of $g$ have even degree,
then there exists an $s\in K[t]$
such that both quadratic forms
\begin{gather}
	\form{1,\pi} \form{1,-\gamma} \form{1,-s}, \label{maaks-vorm1}\\
	\form{1,\pi} \form{1,t g} \form{1,-t s} \label{maaks-vorm2}
.\end{gather}
are isotropic over $K(t)$.
\end{st}

\begin{bew}
Without loss of generality, we may assume that $g$ is a square-free polynomial
(we can divide out squared factors, this does not change the isotropy of \eqref{maaks-vorm2}).
Because of the factor $\form{1,\pi}$ appearing in \eqref{maaks-vorm2},
multiplying $g$ with some power of $\pi$ does not change the isotropy of that quadratic form.
Therefore, we may assume that the leading coefficient $\eps$ of $g$ has valuation zero.

Let $g = \eps \prod_{i=1}^n g_{i}$ be the factorization according to the slopes of $g$.
Write $g_i = \prod_j g_{i j}$, where the $g_{i j}$ are monic and irreducible.
Let $n_i$ denote the degree of $g_i$,
let $m_i = -\varv(g_i(0))/n_i$ denote the slope of $g_i$
and $d_i$ the denominator of $m_i \in \Q$.
Then the degree of every $g_{i j}$ must be a multiple of $d_i$.

Let $N$ be an odd integer which is a multiple of all odd $d_i$ and
large enough such that
\[
	N > \varv(4) / \min_{i \neq j} |m_i - m_j|
	\quad \text{and} \quad
	N > \deg g
.\]

In order to find $s$,
we write $s = \eps \prod_{i=1}^n \prod_j s_{i j}$ and define $s_i := \prod_j s_{i j}$.
The chosen $s_{i j}$ will satisfy the following properties:
\begin{compactenum}[(i)]
\item\label{propfirst} $s_{i j}$ is coprime to $t g$.
\item $s_{i j}$ is monic irreducible with slope $m_i$.
\item\label{evengraad} $s_{i j}$ has even degree.
\item\label{proplast} For all $i$, $\kappa$, $\lambda$ with $i \neq \kappa$, the following equality holds:
\begin{equation}\label{vw-product}
	\Leg{s_i}{g_{\kappa \lambda}} = \Leg{g_i}{g_{\kappa \lambda}}
.\end{equation}
\end{compactenum}

By Theorem~\ref{milnorexactsequence} it suffices to solve the quadratic forms
\eqref{maaks-vorm1} and \eqref{maaks-vorm2} locally
at primes associated with irreducible polynomials in $K[t]$ to solve them globally.
Indeed, if all the second residue maps of such a quadratic form are zero,
then this form is Witt equivalent to an anisotropic form $\varphi$ over $K$.
Since the dimension of $\varphi$ is at most $4$,
the isotropy of \eqref{maaks-vorm1} and \eqref{maaks-vorm2} follows.

The second residue map of (\ref{maaks-vorm1}) is trivially zero,
except at the irreducible factors $s_{i j}$ of $s$.
For each $s_{i j}$, this second residue form is $\form{1,\pi}\form{1,-\gamma}$ over $K[t]/(s_{i j})$.
So property~(\ref{evengraad}) is sufficient to prove that \eqref{maaks-vorm1} is isotropic:
Proposition~\ref{overK} implies that
$\Leg{\gamma}{s_{i j}} = 1$ for each irreducible factor $s_{i j}$ of $s$. 

To prove that $\form{1,\pi}\form{1,t g}\form{1,-t s}$ is isotropic over $K(t)$, we need to consider the second residue forms at $t$ and at each $g_{i j}$ and $s_{i j}$.  The isotropy of these forms is equivalent to the following three conditions:
\begin{align}
	\Leg{s g}{t} &= 1 \label{vw-sg},\\
	\Leg{t s}{g_{i j}} &= 1 \quad \text{for all $i$, $j$}, \label{vw-ts}\\
	\Leg{-t g}{s_{i j}} &= 1 \quad \text{for all $i$, $j$}. \label{vw-tg}
\end{align}

If conditions \eqref{vw-ts} and \eqref{vw-tg} are fulfilled, then \eqref{vw-sg} follows automatically
(because we chose the leading coefficient of $s$ to be equal to $\eps$):
\begin{align*}
	\Leg{s g}{t}
	&= \Leg{\eps}{t} \prod_{i,j} \Leg{s_{i j}}{t} \Leg{\eps}{t} \prod_{i,j} \Leg{g_{i j}}{t}
	= \prod_{i,j} \Leg{-1}{t}^{\deg s_{i j}} \Leg{t}{s_{i j}} \prod_{i,j} \Leg{-1}{t}^{\deg g_{i j}} \Leg{t}{g_{i j}} \\
	&= \Leg{-1}{t}^{\deg s} \Leg{-1}{t}^{\deg g}
		\prod_{i,j} \Leg{-g}{s_{i j}} \prod_{i,j} \Leg{s}{g_{i j}} \\
	&= \prod_{i,j} \left[\Leg{-\eps}{s_{i j}} \prod_{\mu,\nu} \Leg{g_{\mu \nu}}{s_{i j}}\right]
		\prod_{i,j} \left[\Leg{\eps}{g_{i j}} \prod_{\kappa,\lambda} \Leg{s_{\kappa \lambda}}{g_{i j}}\right] \\
	&= \prod_{i,j} \left[\Leg{-\eps}{t}^{\deg s_{i j}} \prod_{\mu,\nu} \Leg{-1}{t}^{\deg g_{\mu \nu} \deg s_{i j}} \Leg{s_{i j}}{g_{\mu \nu}}\right]
		\prod_{i,j} \left[\Leg{\eps}{t}^{\deg g_{i j}} \prod_{\kappa,\lambda} \Leg{s_{\kappa \lambda}}{g_{i j}}\right]
	= 1
\end{align*}
since the degree of $s$ and of $g$ is even.

Using multiplicativity and property~\eqref{vw-product}, we find
\begin{align*}
	\Leg{t s}{g_{i j}}
	&= \Leg{\eps t}{g_{i j}} \prod_\mu \Leg{s_\mu}{g_{i j}}
	= \Leg{\eps t}{g_{i j}} \Leg{s_i}{g_{i j}} \prod_{\mu \neq i} \Leg{g_\mu}{g_{i j}} \\
	&= \Leg{s_i}{g_{i j}} \Leg{t g/g_i}{g_{i j}}
.\end{align*}
Since the degree of each $s_{i j}$ is even, we have
\[
	\Leg{-t g}{s_{i j}}
	= \Leg{t}{s_{i j}} \Leg{-\eps}{s_{i j}}\prod_{\kappa,\lambda} \Leg{g_{\kappa \lambda}}{s_{i j}}
	= \Leg{s_{i j}}{t} \prod_{\kappa,\lambda} \Leg{s_{i j}}{g_{\kappa \lambda}}
.\]
Therefore, the conditions \eqref{vw-ts} and \eqref{vw-tg} become
\begin{align}
	\Leg{s_i}{g_{i j}} &= \Leg{t g/g_i}{g_{i j}}
		\quad \text{for all $i$, $j$}. \label{vw2} \\
	\Leg{s_{i j}}{t} &= \prod_{\kappa,\lambda} \Leg{s_{i j}}{g_{\kappa \lambda}}
		\quad \text{for all $i$, $j$}. \label{vw1}
\end{align}

Now we construct the $s_i$, satisfying properties (\ref{propfirst})--(\ref{proplast})
and conditions (\ref{vw2}) and (\ref{vw1}).
For this, we will have to distinguish two cases,
according to the parity of $d_i$, the denominator of the slope $m_i$.

\textbf{Case 1: $d_i$ is odd.} 

In this case, we will have only one irreducible factor $s_i = s_{i 1}$ with slope $m_i$.
Using \eqref{vw-product} and \eqref{vw2}, we can rewrite the right hand side of condition~(\ref{vw1}) as
\begin{align*}
	\prod_{\kappa,\lambda} \Leg{s_i}{g_{\kappa \lambda}}
	&= \prod_{\kappa\neq i,\lambda} \Leg{s_i }{g_{\kappa \lambda}}\prod_j \Leg{s_i }{g_{i j}}
	= \prod_{\kappa\neq i,\lambda,j} \Leg{g_{i j}}{g_{\kappa \lambda}}\prod_j \Leg{t g/g_i}{g_{i j}}\\
	&= \prod_{\kappa\neq i,\lambda,j} \Leg{-1}{t}^{\deg g_{\kappa \lambda} \deg g_{i j}} \Leg{g_{\kappa \lambda}}{g_{i j}}
		\prod_j \Leg{t g/g_i}{g_{i j}}\\
	&= \prod_{\kappa\neq i} \Leg{-1}{t}^{\deg g_\kappa \deg g_i} \prod_{j} \Leg{\eps^{-1} g/g_i}{g_{i j}}
		\prod_j \Leg{t g/g_i}{g_{i j}}\\
	&= \prod_j \Leg{\eps^{-1} t}{g_{i j}}
	= \prod_j \Leg{\eps^{-1}}{t}^{\deg g_{i j}} \Leg{-1}{t}^{\deg g_{i j}}\Leg{g_{i j}}{t}
	= \Leg{g_i}{t}
.\end{align*}
Therefore, condition~(\ref{vw1}) becomes
\begin{equation}\label{vw1bis}
	\Leg{s_i}{t} = \Leg{g_i}{t}
.\end{equation}
Let
\begin{align*}
	R &= \sum_{\{(\alpha,\beta) \in \Z^2 \mid \beta \geq 0 ~\wedge~ \alpha \geq m_i \beta\}} (\pi^\alpha t^\beta) \Ok \\
	P &= \sum_{\{(\alpha,\beta) \in \Z^2 \mid \beta \geq 0 ~\wedge~ \alpha > m_i \beta \}} (\pi^\alpha t^\beta) \Ok
.\end{align*}
Clearly, $R$ is an $\Ok$-module in $K[t]$, but one can check
that it is actually a subring.
Then $P$ is a prime ideal in $R$.
It is not hard to see that the usual Euclidean division for polynomials
works in $R$, provided we divide by a polynomial whose leading term is not in $P$.

Let $u := \pi^{m_i d_i} t^{d_i}$ and $k = \Ok/(\pi)$. From now on, a line over an element of $R$ denotes reduction modulo $P$.
The quotient ring $R/P$ is $k[\bar{u}]$
(the variable $\bar{u}$ is precisely the reduction of the element $u = \pi^{m_i d_i} t^{d_i}$).
For a polynomial $f(\bar{u}) \in k[\bar{u}]$, we define $\degt f := d_i \deg f$.
This is chosen such that $\degt \overline{f} = \deg f$
for all $f \in R$ with leading term not in $P$.

Define $h_i := \pi^{m_i n_i} g_i$.
Since $g_i$ is monic of degree $n_i$ and slope $m_i$, it follows that $h_i \in R$.
For $g_\mu$ with $m_\mu > m_i$, we have $\pi^{m_\mu n_\mu} g_\mu \in R$
and only the constant term of $\pi^{m_\mu n_\mu} g_\mu$ does not vanish modulo $P$.
For $g_\mu$ with $m_\mu < m_i$,
let $B_i \in 2 d_i \Z$ such that $B_i \geq n_\mu = \deg g_\mu$. 
Then $\pi^{B_i m_i} t^{B_i-n_\mu} g_\mu \in R$
and only the leading term of $\pi^{B_i m_i} t^{B_i-n_\mu} g_\mu$ does not vanish modulo $P$.
So we see that there exist $A, B \in \Z$ with $B$ even
such that $\pi^A t^B g/g_i = \eps \pi^A t^B \prod_{\mu \neq i} g_\mu \in R\setminus P$.
The reduction modulo $P$ of $\pi^A t^B g/g_i$ is of the form $\rho \bar{u}^G$
with $\rho \in k^*$ and $G \geq 0$.

We want to contruct $s_i$ using the ring $R$.
Since the conditions \eqref{vw2} and \eqref{vw1bis}
do not change if we multiply $s_i$ with a multiple of $\pi$,
in reality we will construct $c := \pi^{m_i \deg s_i} s_i \in R$.
We define
\begin{align*}
	a &:= h_i + \pi^{m_i N + A} t^{N+B} g/g_i \in R \\
	b &:= h_i u^{N/d_i+G} = \pi^{m_i (N + d_i G)} t^{N + d_i G} h_i \in R
.\end{align*}
We will construct $c$ of the form $c = a + q b$ for some $q \in R$.
Using the fact that $\Leg{t}{g_{i j}} = \Leg{t^{N+B}}{g_{i j}}$ and $g_{i j} \mid h_i$,
it is clear that $s_i := \pi^{-m_i \deg c} c = \pi^{-m_i \deg c} (a + q b)$ satisfies \eqref{vw2} and \eqref{vw1bis}.
At the same time, we will make sure that $c$
is also of the form $r + \pi^{m_i e} t^e h_i$
for some $e \in 2 d_i \Z$ and $r \in R$ with $\deg r \leq \deg h_i + e - N$.
We claim that for such $c$, the following holds:
\begin{equation}\label{c-over-gkl}
	\Leg{c}{g_{\kappa \lambda}} = \Leg{h_i}{g_{\kappa \lambda}}
	\quad \text{for all $\kappa$, $\lambda$ with $i \neq \kappa$}
.\end{equation}
Indeed, let $\al$ be a root of $g_{\kappa \lambda}$.
If $m_i < m_\kappa = -\varv(\al)$, then by Proposition~\ref{wortelinvullen},
the square class of $c(\al)$ in $K(\al)$ is the same as $\pi^{m_i e} h_i(\al)$.
This implies \eqref{c-over-gkl}.
If $m_i > m_\kappa$, then
the square class of $c(\al)$ in $K(\al)$ is the same as $a(\al)$.
Since $(g/g_i)(\al) = 0$, this also implies \eqref{c-over-gkl}.
Using $\Leg{s_i}{g_{\kappa \lambda}} = \Leg{c}{g_{\kappa \lambda}}$ and $h_i = \pi^{m_i n_i} g_i$,
it is clear that \eqref{c-over-gkl} implies \eqref{vw-product}.

The ideal $P + (u)$ in $R$
contains all $\pi^\alpha t^\beta \in R$, except for $\pi^0 t^0$.
The constant term of $h_i$ has valuation zero, therefore
$(h_i) + P + (u) = (1)$.
Note that $\overline{a} = \overline{h_i} + \rho \bar{u}^{N/d_i + G}$
and $\overline{b} = \overline{h_i} \bar{u}^{N/d_i + G}$.
Since $(\overline{h_i}, \bar{u}) = (1)$ in $R/P$,
Lemma~\ref{ideaal} implies $(\overline{a}, \overline{b}) = (1)$.

Let $N' := N/d_i$.
We apply Theorem \ref{dichtheid} to $k = R/(P + (u))$, $F = k(\bar{u})$, $S_\infty = \{\pr_\infty\}$, then $A = k[\bar{u}]$.
Take
\[
	V_\infty =\{\bar{u}^{e'} + c_{e'-N'}\bar{u}^{e'-N'} + c_{e'-N'-1}\bar{u}^{e'-N'-1} + \dots + c_0 + c_{-1}\bar{u}^{-1} + \dots \mid e' \in 2\Z\}
		\subseteq k((\bar{u}^{-1})) = k_\infty
\]
and $x_\infty = \overline{h_i}$.

Because of Theorem~\ref{dichtheid}, there exist infinitely many $\overline{q_1} \in k[\bar{u}]$
such that $\overline{c} := \overline{a} + \overline{q_1} \overline{b} \in k[\bar{u}]$ is irreducible
and in $\overline{h_i} V_\infty$.
There are infinitely many, so we may assume that $\degt \overline{c} \geq N + \deg b$.

Since $\overline{c} \in \overline{h_i} V_\infty$,
we can write $\overline{c} = \overline{h_i} (\bar{u}^{e'} + \overline{r_0})$,
with $\overline{r_0} \in k((\bar{u}\inv))$ such that $\deg \overline{r_0} \leq e' - N'$.
Let $\overline{r_1} = \overline{h_i} \overline{r_0}  = \overline{c} - \overline{h_i} \bar{u}^{e'} \in k[\bar{u}]$,
then $\degt \overline{r_1} \leq n_i + e - N$.
Choose a lift $r_1 \in R$ of $\overline{r_1}$ such that $\deg r_1 = \degt \overline{r_1}$.
Now let $\tilde{c} = r_1 + \pi^{m_i e} t^e h_i \in R$,
then $\tilde{c}$ is a lift of $\overline{c}$.

Let $q_1 \in R$ be a lift of $\overline{q_1}$.
Lifting the equality $\overline{c} = \overline{a} + \overline{q_1} \overline{b}$ to $R$
yields an error term which is in $P$,
so there exists an $f \in P$ such that
\begin{equation}\label{lift1}
	\tilde{c} + f = a + q_1 b
.\end{equation}
Since the leading term of $b$ is not in $P$, we can do Euclidean division
of $f$ by $b$:
let $f = q_2 b + r_2$ with $\deg r_2 < \deg b$.
Plugging this into \eqref{lift1} gives $\tilde{c} + r_2 = a + (q_1 - q_2) b$.

Reducing the equality $f = q_2 b + r_2$ modulo $P$ gives $0 = \overline{q_2} \overline{b} + \overline{r_2}$.
The leading term of $b$ does not vanish modulo $P$,
so $\degt \overline{r_2} \leq \deg r_2 < \deg b = \degt \overline{b}$.
Since $\overline{r_2}$ is a multiple of $\overline{b}$, it follows that $\overline{r_2} = 0$.

Define $c := \tilde{c} + r_2$, $q := q_1 - q_2$ and $r := r_1 + r_2$.
Then
\[
	c = a + q b = r + \pi^{m_i e} t^e h_i
,\] which is of the required form.
It remains to check that $c$ is irreducible.
Suppose $c$ is reducible in $K[t]$, so $c = c_1 c_2$ with $c_1, c_2 \in K[t]$.
Without loss of generality we can assume that $c_1(0)=1$.
Since $c_2(0) = c(0) = h_i(0)$ is a unit and $c_1$ and $c_2$ have slope $m_i$,
it follows that $c_1, c_2 \in R$.
Therefore, we can reduce modulo $P$ to find $\overline{c} = \overline{c_1} \, \overline{c_2}$.
Now the irreducibility of $\overline{c}$ together with $\degt \overline{c} = \deg c$
implies that $c$ is irreducible.

\textbf{Case 2: $d_i$ is even}.

In this case, every $g_{i j}$ has even degree.
The previous method will not work because every odd degree monomial in $R$ becomes zero in $R/P$.
Instead we will find for each monic irreducible factor $g_{i j}$ of $g_i$ an irreducible $s_{i j}$.
We set
\[
	s_{i j} = g_{i j} + p_{i j}
\]
with $p_{i j} \in R$ with $\deg p_{i j} < \deg g_{i j}$ still to be determined.
The coefficients of $p_{i j}$ will be chosen to have large valuation.
By \cite[Ch.~II, \S\,2, exercice 2]{serre-corpslocaux}, $s_{i j}$ will be irreducible if the valuation
of $p_{i j}$ is sufficiently large (depending on $s_{i j}$).

Next, we want to check that 
\begin{equation}\label{vw-evennoemer}
	\Leg{g_{i j}+p_{i j}}{g_{\mu \nu}} = \Leg{g_{i j}}{g_{\mu \nu}}
	\qquad \text{for all } (\mu,\nu)\neq (i,j)
.\end{equation}
This follows from Lemma~\ref{same-sq-cl} if $\varv(p_{i j}(\al))>\varv(g_{i j}(\al))+\varv(4)$
for each root $\al$ of $g_{\mu \nu}$ with $(\mu,\nu)\neq (i,j)$.
%because then $g_{i j}(\al)$ and $(g_{i j}+p_{i j})(\al)$ are in the same square class. 
Since $g_{i j}(\al)\neq0$ for each root $\al$ of $g_{\mu \nu}$ with $(\mu, \nu)\neq (i,j)$,
we can enforce this condition by taking the coefficients of $p_{i j}$ to have large enough valuation.
Clearly, \eqref{vw-evennoemer} implies that \eqref{vw-product} is satisfied.

Using \eqref{vw-evennoemer}, we rewrite condition~(\ref{vw2}):
\begin{align*}
	\prod_\nu \Leg{s_{i \nu}}{g_{i j}} &= \Leg{t g/g_i}{g_{i j}} \\
	\Leg{s_{i j}}{g_{i j}} \prod_{\nu \neq j} \Leg{g_{i \nu}}{g_{i j}} &= \Leg{t g/g_i}{g_{i j}} \\
	\Leg{p_{i j}}{g_{i j}} &= \Leg{t g/g_{i j}}{g_{i j}} \\
\end{align*}
This condition will be satisfied if we choose $p_{i j}$ such that $p_{i j} \equiv \pi^A t g/g_{i j} \pmod{g_{i j}}$
with $\deg p_{i j} < \deg g_{i j}$ for a large enough $A$.
%such that $p_{i j}\in M$ and $\varv(p_{i j}(\al))>\varv(g_{i j}(\al))+\varv(4)$
%for each root $\al$ of $g_{\mu \nu}$ with $(\mu,\nu)\neq (i, j)$, this condition is fulfilled.
Using the fact that $g_{i j}$ has even degree, the right hand side of \eqref{vw1} becomes
\begin{align*}
	\prod_{\kappa, \lambda} \Leg{s_{i j}}{g_{\kappa \lambda}}
	&= \Leg{s_{i j}}{g_{i j}} \prod_{(\kappa, \lambda) \neq (i,j)} \Leg{g_{i j}}{g_{\kappa \lambda}}
	= \Leg{p_{i j}}{g_{i j}} \prod_{(\kappa, \lambda) \neq (i,j)} \Leg{g_{\kappa \lambda}}{g_{i j}} \\
	&= \Leg{t g/g_{i j}}{g_{i j}} \Leg{\eps^{-1} g/g_{i j}}{g_{i j}}
	= \Leg{\eps^{-1}}{g_{i j}} \Leg{t}{g_{i j}}
	= \Leg{g_{i j}}{t}
.\end{align*}
Since $s_{i j}(0)$ and $g_{i j}(0)$ are in the same square class,
this is equal to $\Leg{s_{i j}}{t}$, therefore \eqref{vw1} is satisfied.
\end{bew}

This easily implies the following corollary,
which corresponds to \cite[Theorem 17]{kim-roush-padic}.

\begin{gev}\label{form-iso}
Let $K$ be a p-adic field with uniformizer $\pi$ and let $\gamma \in K^*$.
Let $g \in K[t]$ with $g(0) \neq 0$.
If all the vertices of the Newton polygon of $g$ have even degree, then the quadratic form
\[
	\form{1,\pi} \form{1, - \gamma, -t, -g}
\]
is isotropic over $K(t)$.
\end{gev}

\begin{bew}
It follows from Theorem~\ref{maaks} that there exists an $s \in K[t]$
such that the quadratic forms \eqref{maaks-vorm1} and \eqref{maaks-vorm2} are isotropic.
By Theorem~\ref{pfister-hyperbolic}, these forms are zero in the Witt ring.

Therefore, in $W(K)$ we also have
\begin{align*}
	0 &= \form{1,\pi} \form{1,-\gamma} \form{1,-s} \orthosum \form{-t} \form{1,\pi} \form{1,t g} \form{1,-t s} \\
	0 &= \form{1,\pi} \form{1, -\gamma, -s, \gamma s, -t, -g, s, -t g s} \\
	0 &= \form{1,\pi} \form{1, -\gamma, -t, -g} \orthosum \form{1,\pi} \form{s} \form{-1, \gamma, 1, -t g}
.\end{align*}
This implies
\[
	\form{1,\pi} \form{1, -\gamma, -t, -g} = - \form{1,\pi} \form{s} \form{\gamma, -t g}
	\quad
	\text{in $W(K)$}
.\]

Since the right hand side has dimension $4$ and the left hand side dimension $8$,
it follows that the left hand side is isotropic.
\end{bew}

\section{Diophantine undecidability}\label{dioph-undec}

In this section we will use the result of the previous section
to give a diophantine definition of the predicate
``$\varv_t(x) \geq 0$'' in $K(t)$.
By Theorem~\ref{denef}, this implies that the existential theory
of $K(t)$ is undecidable.

As before, $K$ denotes a p-adic field with fixed uniformizer $\pi$.
From now on,
fix $\gamma \in K^*$ such that the quadratic form $\form{1,\pi} \form{1,-\gamma}$ is anisotropic over $K$
(for the existence of such a $\gamma$, see for example \cite[Ch.~VI, Corollary 2.15]{lam-qf}).
Throughout this section, we work with the following system of two quadratic forms over $K(t)$:
\begin{align}
	\form{1,\pi}&\form{1, -\gamma, -t, -f} \label{iso-vorm1}\\
	\form{1,\pi}&\form{1, -\gamma, -t, -\gamma f} \label{iso-vorm2}
.\end{align}
with $f \in K(t)$.
First, we prove a theorem which is analogous to \cite[Proposition 7]{kim-roush-padic}.
The quadratic forms above are analogous%
\footnote{The letter $b$ from Kim and Roush corresponds to our $\pi$;
$a$ corresponds to our $\gamma$;
$f$ corresponds to $t \cdot g$
and the forms are multiplied with a factor $\form{t}$.}
to the quadratic forms appearing in the cited Proposition.

\begin{st}\label{oneoftwoanisotropic}
Let $f \in K(t)$ such that $\varv_t(f)$ is odd.
Then one of the quadratic forms \eqref{iso-vorm1} or \eqref{iso-vorm2} is anisotropic over $K(t)$.
\end{st}

\begin{bew}
Let $f_n t^n + f_{n+1} t^{n+1} + \dots$ be the series expansion of $f$ (with $f_n \neq 0$).
By assumption, $n$ is odd.
The first and second residue class forms of \eqref{iso-vorm1} at $t$ are
$\form{1,\pi}\form{1,-\gamma}$ and $\varphi_1 := \form{1,\pi} \form{-1,-f_n}$.
The residue forms of \eqref{iso-vorm2} at $t$ are
$\form{1,\pi}\form{1,-\gamma}$ and $\varphi_2 := \form{1,\pi}\form{-1,-\gamma f_n}$.
By assumption $\form{1,\pi} \form{1,-\gamma}\neq 0$ in $W(K)$.
Since
\begin{align*}
	\varphi_2 - \varphi_1 &= \form{1,\pi}\form{-1,-\gamma f_n, 1, f_n} \\
		&= \form{f_n} \form{1,\pi} \form{1,-\gamma} \neq 0
		\quad \text{in $W(K)$}
,\end{align*}
it follows that $\varphi_1 \neq 0$ or $ \varphi_2 \neq 0$.
Suppose that $\varphi_1 \neq 0$ (the argument for $\varphi_2 \neq 0$ is completely analogous).
Since $\varphi_1$ is a Pfister form, it follows from Theorem~\ref{pfister-hyperbolic}
that $\varphi_1$ is anisotropic.
Both residue forms of \eqref{iso-vorm1} are anisotropic,
therefore \eqref{iso-vorm1} is anisotropic.
\end{bew}

We prove a consequence of Corollary~\ref{form-iso}.
Roughly speaking, we start from a given rational function $h$
and we construct a polynomial such that the vertices of its Newton polygon all have even degree.
It is similar to \cite[Theorem 9]{kim-roush-padic}.

\begin{st}\label{spec-form}
Let $h \in K(t)$ be such that $\varv_\infty(h) \geq -2$ and $\varv_t(h) = 0$.
Then there exists $c \in K$ such that, if we let
\begin{equation}\label{spec-vorm}
	f := h + c t^2
,\end{equation}
both quadratic forms \eqref{iso-vorm1} and \eqref{iso-vorm2} are isotropic over $K(t)$.
\end{st}

\begin{bew}
Since $\varv_\infty(h) \geq -2$ and $\varv_t(h) = 0$,
we can write $h(t) = \frac{h_N(t)}{h_D(t)}$ with $h_N$ and $h_D$ polynomials such that $h_N(0)h_D(0)\neq 0$
and $\deg h_N \leq \deg h_D + 2$.
Multiplying $f$ with $h_D^2$ does not change the isotropy of \eqref{iso-vorm1} and \eqref{iso-vorm2}.
We want to apply Corollary~\ref{form-iso} with $g = f h_D^2$, so
\[g = h_N h_D + c t^2 h_D^2.\]
It is clear that $g(0)\neq 0$ and $\deg g = 2 + 2 \deg h_D$.
We can choose $c \in K$ such that $\varv(c)$ is very low (depending on the coefficients of $h_N$ and $h_D$).
Namely, we can choose $\varv(c)$ so low that the first edge of the Newton polygon of $g$
has vertices of degree $0$ and degree $2$.
Choosing $\varv(c)$ low enough, the remaining vertices of the Newton polygon of $g$
are the vertices of the Newton polygon of $c t^2 h_D^2$, and those also have even degree.
So $g$ satisfies the conditions of Corollary~\ref{form-iso},
and because multiplying $g$ with an element of $K^*$ does not change the degree of the vertices of $g$,
the isotropy of \eqref{iso-vorm1} and \eqref{iso-vorm2} follows. 
\end{bew}

We prove the next theorem similar to \cite[Proposition 4.7]{demeyer-nfpoly},
in which we relate the valuation at $t$ to the isotropy of our system of quadratic forms. 

\begin{st}\label{diophdef}
Let $x \in K(t)$.
Then $\varv_t(x) \geq 0$ if and only if there exists a $c \in K$
such that the quadratic forms (\ref{iso-vorm1}) and (\ref{iso-vorm2}) are isotropic with
\[
	f := \frac{1 + t + t^2 x^3}{1 + t x^3} + c t^2
.\]
\end{st}

\begin{bew}
Define $h_N:= 1 + t + t^2 x^3$, $h_D:= 1 + t x^3$ and $h := h_N/h_D$.

Assume first that $\varv_t(x)\geq 0$.
Then $\varv_t(h_N) = 0$ and $\varv_t(h_D) = 0$ such that $\varv_t(h) = 0$.
If $\varv_\infty(x) \geq 1$, then $\varv_\infty(h_N)= -1$ and $\varv_\infty (h_D) = 0$ such that $\varv_\infty(h)=-1$.
If $\varv_\infty(x) \leq 0$, then $\varv_\infty(h_N)= -2 + 3\varv_\infty(x)$
and $\varv_\infty(h_D) = -1 + 3\varv_\infty(x)$ such that $\varv_\infty(h) = -1$.
In short, if $\varv_t(x)\geq 0$, then $\varv_t(h) = 0$ and $\varv_\infty(h) = -1$.
Theorem~\ref{spec-form} gives us that there exists a $c \in K$
such that \eqref{iso-vorm1} and \eqref{iso-vorm2} are isotropic.

Conversely, assume that $\varv_t(x) \leq -1$.
Then $\varv_t(h_N) = 2 + 3\varv_t(x)$ and $\varv_t(h_D) = 1 + 3\varv_t(x)$ such that $\varv_t(h) = 1$.
It follows that $\varv_t(f) = 1$.
By Theorem~\ref{oneoftwoanisotropic}, for every $c \in K$,
one of the quadratic forms \eqref{iso-vorm1} and \eqref{iso-vorm2} is anisotropic. 
\end{bew}

Since quadratic forms being isotropic is a diophantine condition,
the result in Theorem~\ref{diophdef} is a diophantine definition of the valuation ring at $t$ in $K(t)$,
except for the part ``there exists a $c \in K$''.
We now prove that the constants are diophantine in $K(t)$.
 
\begin{prop}\label{csts-dioph}
$K$ is diophantine in $K(t)$.
\end{prop}

\begin{bew}
Let $E$ be the elliptic curve over $K$ given by the equation $y^2 = x^3-x$.
Let $y \in K$ with $\varv(y)>0$.
We claim that there is an $x \in K$ such that $(x,y)$ lies on $E(K)$.
Let $f(x) = x^3 - x - y^2 \in \Ok[x]$.
Since $\varv(f(0)) = 2\varv(y) > 2\varv(f'(0)) = 0$,
it follows from Theorem~\ref{hensel} that there exists a $b \in \Ok$ such that $f(b)=0$.
So 
\[
	K = \left\{ y_1/y_2 \mid \big(\exists x_1,x_2 \in K\big) \big((x_1,y_1)\in E(K) ~\wedge~ (x_2,y_2)\in E(K) ~\wedge~ y_2 \neq 0 \big)\right\}
.\]

By Hurwitz' Theorem \cite[Ch.~IV, Corollary 2.4]{hartshorne},
the curve $y^2=x^3-x$ admits no rational parametrization, so $E(K(t))=E(K)$.
This means 
\[
	K = \left\{ y_1/y_2 \mid \big(\exists x_1,x_2 \in K(t)\big) \big((x_1,y_1)\in E(K(t)) ~\wedge~ (x_2,y_2)\in E(K(t)) ~\wedge~ y_2 \neq 0 \big)\right\}
;\]
hence $K$ is diophantine in $K(t)$.
\end{bew}

\begin{gev}
In $K(t)$, the relation $\varv_t(x)\geq 0$ is diophantine.
\end{gev}

\begin{bew}
The result follows immediately from Theorem~\ref{diophdef} and Proposition~\ref{csts-dioph}.
\end{bew}

% \begin{opm}
% Our diophantine definition of the predicate ``$\varv_t(x) \geq 0$'' in Theorem~\ref{diophdef} depends
% on the field $K$, since the constants $\pi$ and $\gamma$ appear in it.
% We do not know how to avoid this with an existential definition,
% but we can give a uniform (this means independent of the base field $K$) 
% first order definition for ``$\varv_t(x) \geq 0$''. 
% Let $K$ be a p-adic field, then first of all, there is a uniform first order definition of the valuation ring
% $\Ok$ in $K$ (see for example TODO).
% With this, we can define what it means for an element of $\Ok$ to be a uniformizer. 
% Namely we can define the maximal ideal in $\Ok$ as the elements whose inverse does not belong in $\Ok$,
% and a generator of the maximal ideal as an element that divides every element of the maximal ideal in $\Ok$.
% So we can uniformly define $\pi$ in $K$, and we have that $\varv_t(x) \geq 0$ if and only if 
% for all $\gamma \in K$ there exists a $c \in K$
% such that with
% \[
% 	f := \frac{1 + t + t^2 x^3}{1 + t x^3} + c t^2
% \]
% the quadratic forms (\ref{iso-vorm1}) and (\ref{iso-vorm2}) are isotropic.
% \end{opm}

\bigskip
\begin{acknowledgements}
The authors would like to thank Jan Van Geel
for the many discussions together and the good suggestions he made, and for proofreading the paper.
\end{acknowledgements}

\bibliographystyle{amsplain}
\bibliography{all}

\providecommand{\bysame}{\leavevmode\hbox to3em{\hrulefill}\thinspace}
\providecommand{\MR}{\relax\ifhmode\unskip\space\fi MR }
\providecommand{\MRhref}[2]{%
  \href{http://www.ams.org/mathscinet-getitem?mr=#1}{#2}
}
\providecommand{\href}[2]{#2}
\begin{thebibliography}{10}

\bibitem{bass-milnor-serre}
H.~Bass, J.~Milnor, and J.-P. Serre, \emph{Solution of the congruence subgroup
  problem for {${\rm SL}_{n}\,(n\geq 3)$} and {${\rm Sp}_{2n}\,(n\geq 2)$}},
  Inst. Hautes \'Etudes Sci. Publ. Math. (1967), no.~33, 59--137.

\bibitem{demeyer-nfpoly}
Jeroen Demeyer, \emph{Diophantine sets of polynomials over number fields},
  Proc.\ Amer.\ Math.\ Soc. \textbf{138} (2010), no.~8, 2715--2728.

\bibitem{denef-real}
Jan Denef, \emph{The {D}iophantine problem for polynomial rings and fields of
  rational functions}, Trans.\ Amer.\ Math.\ Soc. \textbf{242} (1978),
  391--399.

\bibitem{engler-prestel}
Antonio Engler and Alexander Prestel, \emph{Valued fields}, Springer Monographs
  in Mathematics, Springer, 2005.

\bibitem{hartshorne}
Robin Hartshorne, \emph{Algebraic geometry}, Graduate Texts in Mathematics,
  no.~52, Springer, 1977.

\bibitem{kim-roush-padic}
Ki~Hang Kim and Fred Roush, \emph{Diophantine unsolvability over $p$-adic
  function fields}, J.\ Algebra \textbf{176} (1995), no.~1, 83--110.

\bibitem{lam-qf}
Tsit-Yuen Lam, \emph{Introduction to quadratic forms over fields}, Graduate
  Studies in Mathematics, no.~67, American Mathematical Society, 2005.

\bibitem{neukirch-az}
J\"urgen Neukirch, \emph{Algebraische {Z}ahlentheorie}, Springer-Verlag, 1992.

\bibitem{pheidas-zahidi}
Thanases Pheidas and Karim Zahidi, \emph{Undecidability of existential theories
  of rings and fields: a survey}, Hilbert's Tenth Problem: Relations with
  Arithmetic and Algebraic Geometry ({Ghent, 1999}) (Denef et~al., eds.),
  Contemp.\ Math., vol. 270, 2000, pp.~49--105.

\bibitem{scharlau-qhf}
Winfried Scharlau, \emph{Quadratic and hermitian forms}, Grundlehren Math.\
  Wiss., no. 270, Springer Berlin, 1985.

\bibitem{serre-corpslocaux}
Jean-Pierre Serre, \emph{Corps locaux}, Actualit\'es scientifiques et
  industrielles, Hermann, 1980.

\end{thebibliography}

\end{document}